\documentclass[11pt]{amsart}
\usepackage{graphicx}
\usepackage{amsmath}
\usepackage{xcolor}
\usepackage[active]{srcltx}
\usepackage{fancyhdr}

\usepackage{hyperref}

\makeatletter
\renewcommand*\subjclass[2][2000]{%
  \def\@subjclass{#2}%
  \@ifundefined{subjclassname@#1}{%
    \ClassWarning{\@classname}{Unknown edition (#1) of Mathematics
      Subject Classification; using '1991'.}%
  }{%
    \@xp\let\@xp\subjclassname\csname subjclassname@#1\endcsname
  }%
}
\makeatother
\usepackage{enumerate,url,amssymb,  mathrsfs}%, pdfsync}% ,showkeys, pdfsync}

\newtheorem{theorem}{Theorem}[section]
\newtheorem{lemma}[theorem]{Lemma}
\newtheorem*{lemma*}{Lemma}

\theoremstyle{definition}

\theoremstyle{remark}

\numberwithin{equation}{section}

\def\XXint#1#2#3{{\setbox0=\hbox{$#1{#2#3}{\int\limits}$}
\vcenter{\hbox{$#2#3$}}\kern-.5\wd0}}

\begin{document}

\title{Quasiconformal harmonic mappings between the unit ball and a spatial domain with $C^{1,\alpha}$ boundary}
%\date{11 October, 2005}

\address{University of Montenegro, Faculty of Natural Sciences and
Mathematics, Cetinjski put b.b. 81000 Podgorica, Montenegro}
\email{antondj@ucg.ac.me}

\email{davidk@ucg.ac.me}

 \author{Anton Gjokaj and David Kalaj}

\begin{abstract}
We prove the following. If $f$ is a harmonic quasiconformal mapping between the unit ball in $\mathbb{R}^n$ and a spatial domain with $C^{1,\alpha}$ boundary, then $f$ is Lipschitz continuous in $B$. This generalizes some known results for $n=2$ and improves some others in higher dimensional case.
\end{abstract}
\maketitle

\section{Introduction}

For $n > 1$, let $\mathbb{R}^n$ be the standard Euclidean space with the norm $|x|=(x_1^2+\ldots+x_n^2)^{\frac{1}{2}}$, where $x=(x_1,\ldots,x_n).$ We denote the unit ball $\{x \in \mathbb{R}^n: |x|<1\}$ by $B$, and its boundary, the unit sphere $\{x \in \mathbb{R}^n: |x|=1\}$ by $S.$ \\

Let $U\subset \mathbb{R}^n$ be a domain. We say $f=(f_1,\ldots,f_n):U\to \mathbb{R}^n$ is a harmonic mapping if the functions $f_j$ are harmonic real mappings, i.e. satisfy the $n$-dimensional Laplace equation
\[ \Delta f_j=\sum_{i=1}^{n} D_{ii}f_j=0.\]
Let
\[ P(x,\xi)=\frac{1-|x|^2}{|x-\xi|^n}\]
be the Poisson kernel for $B$, where $x\in B$, $\xi \in S$, and
\[ P[u](x)=\int\limits_{S} P(x,\xi) u(\xi) d\sigma (\xi) \]
the Poisson integral of continuous function $u$ on $S$, where $\sigma$ denotes the normalized surface-area measure on $S$. Then $P[u](x)$ is continuous on $\overline{B}$ and harmonic on $B$. Since we will focus on continuous function $u$ on $\overline{B}$, that are harmonic on $B$,  then we will usually express them using the Poisson integral as
\[ u=P[u|_{S}](x).\]

A homeomorphism $f:U\to V$, where $U,V$ are domains in $\mathbb{R}^n$, will be called $K$ quasiconformal (see \cite{Vaisala}) $(K\geq 1)$ if $f$ is absolutely continuous on lines (i.e. absolutely continuous in almost every segment parallel to some of the coordinate axes and there exist partial derivatives which are locally $L^{n}$ integrable in $U$) and
\[ |\nabla f(x)|\leq K l(\nabla f(x)),  \]
for all points $x\in U$, where
\[ l(\nabla f(x))=\inf\{|f'(x) h|: |h|=1\}. \]
\\[0.2cm]

A function $\Phi:U\subset \mathbb{R}^n \to \mathbb{R}$ is said to be $\mu$-H\"{o}lder continuous, $\Phi\in C^{\mu}(U)$ if
\[ \sup_{x,y\in U, x\neq y} \frac{|\Phi(x)-\Phi(y)|}{|x-y|^{\mu}}<\infty.\]
Similarly, one defines the class $C^{1,\mu}(U)$ to consist of all functions $\Phi\in C^1(U)$ such that $\nabla \Phi \in C^{\mu}(U)$.
The above two definitions extends in a natural way to the case of vector-valued mappings.

We say that a domain $\Omega \subset \mathbb{R}^n$ has $C^{1,\alpha}$ boundary if there is a  $C^{1,\alpha}$ diffeomorphism $G:\overline{B}\to \overline{\Omega}$. 
\\

Pavlovi\'c in \cite{Pavlovic1} showed that harmonic quasiconformal mappings of the unit disk in $\mathbb{R}^2$ onto itself are bi-Lipschitz mappings. From then, several important results have been obtained regarding harmonic quasiconformal mappings in $\mathbb{R}^2$ and the Lipschitz continuity. The second author in \cite{Kalaj1} proved that every quasiconformal harmonic mapping between Jordan domains with $C^{1,\alpha}$ boundaries is Lipschitz continuous on the closure of domain. The result in \cite{Kalaj1} was extended in \cite{Kalaj2} for Jordan domains with only Dini's smooth boundaries. Lately, in \cite{Kalaj6} it was proved the H\"{o}lder continuity (but in general, Lipschitz continuity does not hold) of a harmonic quasiconformal mapping between two Jordan domains  having only $C^{1}$ boundaries. Other important results for $n=2$ with different conditions and settings can be found in \cite{Arsenovic}, \cite{BozinMat}, \cite{Gehring}, \cite{Kalaj4}, \cite{Kalaj5},  \cite{KalajMat}, \cite{KalajPavlovic}, \cite{Manojlovic}, \cite{Martio}, \cite{Martio1}, \cite{Partyka}, \cite{Partyka1} and in their references. \\
For higher dimensional case there are some important results also (see e.g. \cite{Astala},\cite{Kalaj3}, \cite{KalajZl}, \cite{MatVuo}). In \cite{Kalaj3} it was proven that a quasiconformal mapping of the unit ball onto a domain with $C^2$ smooth boundary, satisfying Poisson differential inequality, is Lipschitz continuous. This implies that harmonic quasiconformal mappings from unit ball $B$ to $\Omega$ with $C^2$ boundary are Lipschitz continuous.
This was also proved by Astala and Manojlovic in \cite{Astala} using a slight modification of the following statement also proved there: a harmonic $K$-quasiconformal mapping from $B$ to $B$ is Lipschitz with the Lipschitz constant depending on the value of $K$, dimension of $n$ and dist$(f(0), S)$.
\\[0.3cm]
%%%%%%%%%%%%%
Our main result generalizes the result in \cite{Kalaj1} and improves the mentioned corollaries in \cite{Astala} and \cite{Kalaj3}. It reads as follow.
\begin{theorem} \label{glavna}
Let $f:B \to \mathbb{R}^n$ be a quasiconformal harmonic (qch) mapping, $f(B)=\Omega$, and $\partial \Omega \in C^{1,\alpha}$. Then $f$ is Lipschitz continuous in $B$.
\end{theorem}

 The proof of the corresponding result for 2-dimensional case in \cite{Kalaj1} uses conformal mappings, however conformal mappings in higher-dimensional setting are very rigid, and this is why we need to find another way to deal with the proof of Theorem~\ref{glavna}. The initial idea lies on the following simple approach. Let $\eta \in S$ and $f(\eta)=q \in \partial \Omega$. We can suppose that $q=0$ and the tangent plane of $q$ at $\partial \Omega$ is $x_n=0$. This can be obtained in the following way: Using a isometry $L$ we can postcompose $f$ such that we get a function $\tilde{f}$ from $B$ to $\Omega'$, $\tilde{f}(\eta)=0$ and the tangent plane of this point on $\partial \Omega'$ is $x_n=0.$ Observe that $\tilde{f}$ is also harmonic and quasiconformal, because it is composed by a isometry. The Lipschitz continuity for function $\tilde{f}$ would yield the proof of this property for the function $f$ also, because the isometry preserves the distances.

The proof is given in Section $3.$  It uses an iteration procedure.  Before that, in next section, we give some basic preparations through Theorems $\ref{calpha}$-$\ref{pavlovic}$.

\section{Auxiliary results}

The next theorem is of general interest; on the other side it plays an important role in proving Theorem $\ref{glavna}$. Some versions of it for $n=2$ can be found in \cite{Goluzin} and \cite{Nitsche}.

\begin{theorem}
\label{calpha}
Let $u:\overline{B}\subset\mathbb{R}^n \to \mathbb{R}$, $n \geq 3,$ be a real harmonic function, $\eta \in S$. Assume that $|u(\xi)-u(\eta)|\leq M |\xi-\eta|^{\mu}$, $\forall \xi \in S$,  for some $\mu\in (0,1)$. Then we have $C=C(M,\mu,n)$ such that
\[ |\nabla u(x)| (1-|x|)^{1-\mu}\leq C,\] where $x=r\eta$, $r\in [0,1).$

\end{theorem}

\begin{proof} Throught the proof, the constant $C$ can change its value. Using the Poisson integral formula we have
\[ u(x)=\int\limits_{S} \frac{1-|x|^2}{(1+|x|^2-2\langle \xi, x \rangle)^{\frac{n}{2}}} u(\xi) d\sigma (\xi).\]
Observe that
\begin{equation} \label{000001}
\nabla u(x) =\int\limits_{S} Q(x,\xi) u(\xi) d\sigma (\xi),
\end{equation}
where
\begin{equation} \label{002}
\begin{split}
Q(x,\xi)&=\frac{(-2x)(1+|x|^2-2\langle \xi, x \rangle)^{\frac{n}{2}}-n(1-|x|^2)(1+|x|^2-2\langle \xi, x \rangle)^{\frac{n}{2}-1}(x-\xi) }{(1+|x|^2-2\langle \xi, x \rangle)^{n}}\\
&=\frac{(-2x)(1+|x|^2-2\langle \xi, x \rangle)-n(1-|x|^2)(x-\xi) }{(1+|x|^2-2\langle \xi, x \rangle)^{\frac{n}{2}+1}} \\
&=\frac{(-2x)(1+|x|^2-2\langle \xi, x \rangle)-n(1-|x|^2)(x-\xi) }{(1+|x|^2-2\langle \xi, x \rangle)}\cdot \frac{1}{(1+|x|^2-2\langle \xi, x \rangle)^{\frac{n}{2}}}.
\end{split}
\end{equation}
Let $h\in R^n$ be an arbitrary vector.
Then
\begin{equation} \label{001}
\langle \nabla u(x), h \rangle =\int\limits_{S} \langle Q(x,\xi), h \rangle u(\xi) d\sigma (\xi).
\end{equation}
Since $(\ref{001})$ is true for every harmonic function $u:\overline{B} \to \mathbb{R}$, taking the constant function $u(\eta)$, we get
\begin{equation}
0=\int\limits_{S} \langle Q(x,\xi), h \rangle u(\eta) d\sigma (\xi),
\end{equation}
which, together with $(\ref{001})$, gives us
\begin{equation}  \label{003}
\langle \nabla u(x), h \rangle =\int\limits_{S} \langle Q(x,\xi), h\rangle [u(\xi)-u(\eta)] d\sigma (\xi).
\end{equation}
On the other side
\begin{equation}\label{004}
\begin{split}
& \left|\frac{-2\langle x, h \rangle(1+|x|^2-2\langle \xi, x \rangle)-n(1-|x|^2)\langle x-\xi, h \rangle }{(1+|x|^2-2\langle \xi, x \rangle)}\right| \\
& \leq 2|x||h|+n\frac{(1-|x|^2)|x-\xi||h|}{|x-\xi|^2}\leq  \\
& = 2|x||h| +2n|h| \frac{1-|x|}{|x-\xi|} \leq (2+2n)|h|.
\end{split}
\end{equation}
In the last inequality it is used the fact that $1-|x|\leq|x-\xi|$, which is obviously true from the geometrical point of view, but it is also equivalent to $\langle \xi, x \rangle \leq |x|$ (Cauchy-Schwarz inequality).

From  $(\ref{002}), (\ref{003}), (\ref{004})$ we get
\begin{equation}
|\langle \nabla u(x), h \rangle|\leq (2n+2)|h| \int\limits_{S} \frac{|u(\xi)-u(\eta)|}{(1+|x|^2-2\langle \xi, x \rangle)^{\frac{n}{2}}} d\sigma (\xi)
\end{equation}
As $h$ was taken arbitrary, then
\begin{equation}
|\nabla u(x)|\leq (2n+2) \int\limits_{S} \frac{|u(\xi)-u(\eta)|}{(1+|x|^2-2\langle \xi, x \rangle)^{\frac{n}{2}}} d\sigma (\xi),
\end{equation}
which is equivalent to
\begin{equation} \label{330}
\begin{split}
|\nabla u(r\eta)| &\leq (2n+2)\int\limits_{S} \frac{|u(\xi)-u(\eta)|}{(1+r^2-2r\langle \xi, \eta \rangle)^{\frac{n}{2}}} d\sigma (\xi)\\
&=(2n+2)\int\limits_{S} \frac{|u(\xi)-u(\eta)|}{((1-r)^2+r|\xi-\eta|^2)^{\frac{n}{2}}} d\sigma (\xi),
\end{split}
\end{equation}
where $x=r\eta$, $r=|x| \in [0,1).$

Using the condition of the theorem we get
\begin{equation} \label{33}
|\nabla u(r\eta)| \leq M(2n+2) \int\limits\limits_{S} \frac{|\xi-\eta|^{\mu}}{((1-r)^2+r|\xi-\eta|^2)^{\frac{n}{2}}}d\sigma(\xi).  \\
\end{equation}
\\[0.3cm]
%\textbf{1st case}: $|x|=r\geq\frac{1}{2}.$ \\
%\[|\nabla u(r\eta)| \leq C \int\limits\limits_{S} \frac{|\xi-\eta|^{\mu}}{((1-r)^2+\frac{1}{2}|\xi-\eta|^2)^{\frac{n}{2}}}d\sigma(\xi).  \]
Because of the symmetry, it is enough to show the required inequality for $\eta=(1,0,\ldots,0)$.  \\[0.2cm]
\textbf{1st case} $r=|x| \geq \frac{1}{2}.$
\\
As the integrand function in $(\ref{33})$ depends only on the first coordinate of $\xi$, we use the following representation (\cite{Axler}, Appendix A5):
%\[ \int_{S_{n-1}} f(\xi) dS_{n-1} = \int_{-1}^{1} (1-x^2)^{\frac{n-3}{2}}\int_{S_{n-2}}f(x,\sqrt{1-x^2}\eta)dS_{n-2}dx.\]

\[ |\nabla u(r\eta)| \leq M (2n+2) C_1\int_{-1}^{1} \int\limits\limits_{S_{n-2}} \frac{(2-2x)^{\frac{\mu}{2}}}{((1-r)^2+r(2-2x))^{\frac{n}{2}}} (1-x^2)^{\frac{n-3}{2}} d\sigma_{n-2}(\zeta)dx,  \]
where $\sigma_{n-2}$ denotes the respective normalized surface-area measure on the unit sphere $S_{n-2}$ in $\mathbb{R}^{n-1}$. The constant $C_1$ depends on $n$ and the volumes of the unit balls in $\mathbb{R}^n$ and $\mathbb{R}^{n-1}$. From this, it follows
\begin{equation}
\begin{split} 
 |\nabla u(r\eta)| & \leq 
 C \int\limits\limits_{S_{n-2}} d\sigma_{n-2}(\zeta) \int_{-1}^{1} \frac{(2-2x)^{\frac{\mu}{2}}}{((1-r)^2+r(2-2x))^{\frac{n}{2}}} (1-x^2)^{\frac{n-3}{2}}dx \\
&\leq C \int_{-1}^{1}  \frac{(2-2x)^{\frac{\mu}{2}}}{((1-r)^2+r(2-2x))} \frac{2^{\frac{n-3}{2}}(1-x)^{\frac{n-3}{2}}}{{((1-r)^2+r(2-2x))^{\frac{n-2}{2}}}  } dx \\
& = C 2^{\frac{\mu}{2}} 2^{\frac{n-3}{2}} \int_{-1}^{1}  \frac{(1-x)^{\frac{\mu-1}{2}}}{((1-r)^2+r(2-2x))} \left(\frac{1-x}{(1-r)^2+r(2-2x) }\right)^{\frac{n-2}{2}} dx. 
\end{split} 
\end{equation}
Since $r\geq \frac{1}{2}$, we easily get
\[ \frac{1-x}{(1-r)^2+r(2-2x)}\leq \frac{1-x}{(1-r)^2+(1-x)}\leq 1,\]
so \begin{equation}\label{299}
|\nabla u(r\eta)| \leq C \int_{-1}^{1}  \frac{(1-x)^{\frac{\mu-1}{2}}}{(1-r)^2+(1-x)}dx. 
\end{equation}

First, using the substitution $1-x=t^2$, then $s=\frac{t}{1-r}$, we have
\[ 
|\nabla u(r\eta)| \leq C \int_{0}^{\sqrt{2}}  \frac{2t^{\mu}}{(1-r)^2+t^2}dt=
 C \frac{(1-r)^{\mu}}{(1-r)^2} \int_{0}^{\frac{\sqrt{2}}{1-r}}  \frac{2s^{\mu}}{1+s^2}(1-r)ds,
\]
so
\[
|\nabla u(r\eta)| \leq C (1-r)^{\mu-1} \int_{0}^{\infty} \frac{s^{\mu}}{1+s^2}ds. \]

As the last integral converges we finally have
\begin{equation} \label{006} |\nabla u(r\eta)| (1-r)^{1-\mu} \leq C, \end{equation} for $r\in [\frac{1}{2},1)$, where $C$ depends on $M, \mu$ and $n$ only. \\[0.5cm]

\textbf{2nd case} $r=|x| < \frac{1}{2}$
\\ 
Here the proof is quite straightforward. Since
\begin{equation} \label{018} \frac{|\xi-\eta|^{\mu}(1-r)^{1-\mu}}{{((1-r)^2+r|\xi-\eta|^2)^{\frac{n}{2}}}}< \frac{2^{\mu}\cdot 1}{(\frac{1}{2})^n}=2^{n+\mu}, \end{equation}
using $(\ref{33})$ we get
\begin{equation} \label{007} |\nabla u(r\eta)| (1-r)^{1-\mu} \leq M (2n+2) 2^{n+\mu}.\end{equation}
\\[0.3cm]
We conclude that the inequality is true for all $r\in[0,1)$, with the final $C$ being the larger of the obtained constants on the RHS of $(\ref{006})$ and $(\ref{007})$.
\end{proof}

The idea of the proof in section $3$ will be based on obtaining locally the $C^{\mu}$ condition of $f$ on the unit sphere for $\mu<1$, by increasing $\mu$. In relation to a fixed point $\eta \in S$ this will, in one moment, give us a similar inequality as the one from Theorem $\ref{calpha}$, but for $\mu>1$. So, on this step, we need a different version of the previuos statement which is given in the following theorem. However, the proof of it is very similar to the proof of the previous one.

\begin{theorem}
\label{calpha2}
Let $u:\overline{B}\subset\mathbb{R}^n \to \mathbb{R}$, be a harmonic function, $\eta \in S$. Assume that $|u(\xi)-u(\eta)|\leq M |\xi-\eta|^{\mu}$, $\forall \xi \in S$,  for some $\mu > 1$. Then we have $C=C(M,\mu,n)$ such that
\[ |\nabla u(r\eta)| \leq C,\] for every $r\in[0,1).$

\end{theorem}
\begin{proof}
The proof of the theorem for $r\in[\frac{1}{2},1)$ is identical to the previous theorem until ($\ref{299}$).
\begin{equation*} 
\int_{-1}^{1}  \frac{(1-x)^{\frac{\mu-1}{2}}}{(1-r)^2+(1-x)}dx \leq \int_{-1}^{1} (1-x)^{\frac{\mu-3}{2}}dx= \frac{2^\frac{\mu+1}{2}}{\mu-1}
\end{equation*}
shows that the inequality is true. \\[0.1cm]

For $r\in [0,\frac{1}{2})$, similar to $(\ref{018})$ we see that
\[ \frac{|\xi-\eta|^{\mu}}{{((1-r)^2+r|\xi-\eta|^2)^{\frac{n}{2}}}} \]
is bounded, so therefore again from $(\ref{33})$ we have our inequality.
\end{proof}
%%%%%
The next celebrated theorem will also be used.  The proof can be found in \cite{Mori}.
\begin{theorem} \textbf{(Mori's theorem)} \label{mori}
Let $g$ be a $K$-quasiconformal mapping of ${B}$ onto ${B}$, $n\geq2$, with $g(0)=0$. Then
\[ |g(x)-g(y)|\leq M(n,K) |x-y|^{\beta},\]
for all $x,y \in B$, where $\beta=K^{\frac{1}{1-n}}$.

\end{theorem}

We collect now the following useful result. The proof can be found in \cite{Pavlovic}. We will formulate it in the form which corresponds to our notation and use.

\begin{theorem} \label{pavlovic} Let $u$ be a real harmonic function on $\overline{B}$ and  $\mu \in (0,1)$ such that
\begin{equation} \label{010} \left||u(r\eta)|-|u(\eta)|\right|\leq C (1-r)^{\mu}, \hspace{0.5cm} \text{$\forall r\in [0,1), \eta \in S$}, \end{equation}
where $C$ is independent of $r$ and $\eta$, then $u$ is $\mu$-H\"{o}lder continuous in $\overline{B}$, i.e.:
\[ |u(x)-u(y)|\leq M |x-y|^{\mu}, \]
for all $x,y\in \overline{B}$. \end{theorem}

Using the previous theorem we can easily prove the following lemma.
\begin{lemma} \label{015}  Let $u$ be a real harmonic function on $\overline{B}$ and $\mu \in (0,1)$ such that
\[ |\nabla u(r\eta)| \leq C (1-r)^{\mu-1}, \hspace{0.5cm} \text{ $\forall r\in(0,1), \eta \in S$,}\]
where $C$ does not depend on $r$ and $\eta$, then $u$ is $\mu$-H\"{o}lder continuous in $\overline{B}$.
\end{lemma}

\begin{proof}
To prove this lemma, from Theorem $\ref{pavlovic}$, it is sufficient to show the relation $(\ref{010})$.
%\begin{equation} \label{011} |u(r\eta)-u(\eta)| \leq C (1-r)^{\mu}, \hspace{0.5cm} \text{$ \forall r \in [0,1), \eta \in S$}. \end{equation}

We have
\begin{equation}
u(\eta)-u(r\eta)=\int\limits_{\gamma_r} D_1 u dx_1 +\ldots + D_n u dx_n,
\end{equation}
where $\gamma_r$ is the radial segment with endpoints $r\eta$ and $\eta$.
\\
Therefore, we have
\begin{equation}
\begin{split}
 \left||u(r\eta)|-|u(\eta)|\right|& \leq |u(r\eta)-u(\eta)| \leq \int\limits_{r}^{1} |\langle \nabla u(t\eta), \eta \rangle| dt \\
& \leq  C \int\limits_{r}^{1} (1-t)^{\mu-1}dt \leq C \frac{(1-r)^{\mu}}{\mu}.
\end{split}
\end{equation}

\end{proof}

%%%%%%%

\section{Proof of the main result - Theorem \ref{glavna}}
\begin{proof}
First, let we prove the H\"{o}lder continuity of $f=(f_1,\ldots,f_n)$. Indeed, let $G$ be a quasiconformal diffeomorphism (recall that $\Omega$ has $C^{1,\alpha}$ boundary) from $B^n$ to $\Omega$ which is Lipschitz continuous mapping up to the boundary, such that $G(0)=f(0)$. Then the mapping $g=G^{-1} \circ f$ is a $K'$ quasiconformal mapping (as a composition of two quasiconformal mappings) of $B$ onto $B$, where $g(0)=0$. According to Mori's theorem \ref{mori}, there exists a constant $M_1(n,K')$ such that
\[|g(x)-g(y)|\leq M_1(n,K') |x-y|^{K'^{\frac{1}{1-n}}},\]
for all $x,y \in B^n$.

Since $f=G\circ g$, then $f$ satisfies a similar inequality, being a composition of Lipschitz and H\"{o}lder continuous functions:
\begin{equation} \label{mori1}|f(x)-f(y)|\leq C_1 |x-y|^{\beta}, \end{equation}
for all $x,y \in \overline{B^n}$, where $\beta \in(0,1)$, and the constant $C_1$ depends on $M_1$ and the Lipschitz constant of $G$.
\\

In view of the remark after the formulation of Theorem $\ref{glavna}$, there exists a neighbourhood $\mathcal{O}$ of the origin in $R^{n-1}$ which is the projection of $\partial \Omega \cap B(0,\rho)$ in $R^{n-1}$ and a $C^{1,\alpha}$ function $\Phi:\mathcal{O}\to\mathbb{R}$ such that $\partial \Omega \cap B(0,\rho)$ can be expressed as the graph of $\Phi$, i.e. points of $\partial \Omega \cap B(0,\rho)$ are of the form:
\begin{equation} \label{5}  (\zeta_1,\ldots ,\zeta_{n-1}, \Phi(\zeta_1,\ldots, \zeta_{n-1})), \end{equation} 
where $(\zeta_1,\ldots ,\zeta_{n-1}) \in \mathcal{O}$.  

The function $\Phi$ has the properties $\Phi(0,\ldots, 0)=0$ and $D_j\Phi(0,\ldots,0)=0$, for all $j \in \{1,2,\ldots, n-1\}$, and

\begin{equation} \label{3}
|\nabla \Phi(\zeta)-\nabla \Phi(\omega)|\leq C_2 |\zeta-\omega|^{\alpha}. \end{equation}
The constant $C_2$ is the same for all points $q\in \partial \Omega$, because of the $C^{1,\alpha}$ condition of $\partial \Omega$.
%The radius $\rho$ and the constant $C_2$ can be choosen small enough for all points $q\in \partial \Omega$, such that after an isometry $L$ this $\rho$-neighbourhood of $L(q)=0$ is a graphic of a function as in $(\ref{5})$, where $\Phi_q$ is a $C^{1,\alpha}$ function with the universal constant $C_2$ and $\Phi'_q(0)=0$. This will be used later.
\\ Also,
\begin{equation} \label{1}|\Phi(\zeta)-\Phi(\omega)|=|\langle \nabla \Phi(c),\zeta-\omega\rangle|\leq |\nabla \Phi(c)||\zeta-\omega|, \end{equation}
where $c$ belongs to the segment $[\zeta,\omega]$.

Using $(\ref{3})$ we get
\begin{equation} \label{2}
\begin{split}
 |\nabla \Phi(c)|& \leq |\nabla \Phi(\zeta)|+|\nabla \Phi(c)-\nabla \Phi(\zeta)|
\\ & \leq C_2(|\zeta|^\alpha+|c-\zeta|^{\alpha}) \leq C_2 ( |\zeta|^{\alpha}+|\zeta-\omega|^{\alpha}),
\end{split}
\end{equation}

\begin{equation} \label{2'}
\begin{split}
|\nabla \Phi(c)|& \leq |\nabla \Phi(\omega)|+|\nabla \Phi(c)-\nabla \Phi(\omega)|
\\ &\leq C_2( |\omega|^\alpha+|c-\omega|^{\alpha}) \leq C_2 ( |\omega|^{\alpha}+|\zeta-\omega|^{\alpha}), \\
\end{split}
\end{equation}
which yields
\[ |\nabla \Phi(c)|\leq C_2 \min \{\{|\zeta|^{\alpha}, |\omega|^{\alpha}\}+|\zeta-\omega|^{\alpha}\} .\]
Therefore, from $(\ref{1})$ we have:
\begin{equation} \label{4} |\Phi(\zeta)-\Phi(\omega)|\leq C_2|\zeta-\omega|(\min\{|\zeta|^{\alpha}, |\omega|^{\alpha}\}+|\zeta-\omega|^{\alpha}),  \end{equation}
for all $\zeta, \omega$ in $\mathcal{O}.$
\\
Let $F=(F_1,\ldots,F_n)=f|_{S}$ or $P[F]=f$. Notice that $F$ is also $C^{\beta}$ in $S$. We will use the notation $\tilde{F}(\xi)=(F_1(\xi),\ldots, F_{n-1}(\xi)).$ $\tilde{F}$, as $F$, also satisfies $(\ref{mori1})$. In view of $(\ref{5})$ we have that in a small neighbourhood of $\eta$ in $S$, $F_n$ is of the form \[ F_n(\xi)=\Phi(F_1(\xi),\ldots,F_{n-1}(\xi)).\]
\\[0.01cm]
We may also assume that this neighbourhood of $\eta$ is of the form $V(\eta)=B(\eta, \delta) \cap S$, where $\delta$ is small enough positive constant for all $q\in \partial \Omega$. Indeed, let $\tilde{U}(q)=B(q,r_q)\cap \partial \Omega$ be the neighbourhood of $q$ in $\partial \Omega$ such that after the isometry $L_q$ (the one that sends $q$ to $0$ and which makes the plane $x_n=0$ the tangent plane of $\partial \Omega$ at point $0$), $L_q(\tilde{U}(q))$ is the neighbourhood of $0$ which is the graph of a function as in $(\ref{5})$. Furthermore, we can choose $r_q$ small enough, such that for every point  $p \in \tilde{U}(q)$, the image of  $\tilde{U}(q)$ under the respective isometry $L_p$ is a graph of a function.
\\
Observe now $U(q)=B(q,\frac{r_q}{2})\cap \partial \Omega$. The collection $\{U(q)\}_{q\in \partial \Omega}$ is a cover of $\partial \Omega$. As $\partial \Omega$ is compact, there exists a finite subcollection $\{U(q_k)\}_{k=1}^{m}$ which covers $\partial \Omega$. Let $\rho=\min \{ \frac{r_{q_1}}{2},\ldots, \frac{r_{q_m}}{2}\}$. Since $F$ is continuous on a compact, there is a $\delta>0$ such that if $|\xi_1-\xi_2|<\delta$, $\xi_1, \xi_2 \in S,$
then $|F(\xi_1)-F(\xi_2)|<\frac{\rho}{2}$. \\ This ensures that the image of every $V(\eta)=B(\eta, \delta) \cap S$ under $F$ is contained in a $B(q_j,r_{q_j})\cap \partial \Omega=\tilde{U}(q_j)$, and further, after the mentioned isometry is done, this image is the graph of a function as in $(\ref{5})$.
\\[0.2cm]

We get back to our fixed $\eta$, such that $f(\eta)=0$. Now
\begin{equation} \label{6}
\begin{split}
|F_n(\xi)-F_n(\eta)|&= |\Phi(\tilde{F}(\xi))-\Phi(0)| \\
&\leq C_2 |\tilde{F}(\xi)|(\min\{|\tilde{F}(\xi)|^{\alpha}, 0\}+|\tilde{F}(\xi)-0|^{\alpha})\\
&=C_2 |\tilde{F}(\xi)|^{1+\alpha}
\leq C_1^{1+\alpha} C_2|\xi-\eta|^{(1+\alpha)\beta},
\end{split}
\end{equation}
for all $\xi \in V(\eta)$.
The function $F_n$ is bounded, because $F=f|_{S}$ is bounded ($|F(\xi)|\leq \widetilde{M}$, for all $\xi\in S$), so if $\xi \in S \backslash V(\eta)$ then
\begin{equation} \label {7}
|F_n(\xi)-F_n(\eta)|\leq 2\widetilde{M} \leq \frac{2\widetilde{M}}{\delta^{(1+\alpha)\beta}} |\xi-\eta|^{(1+\alpha)\beta}.
\end{equation}
Taking $M=\max\{C_1^{1+\alpha}C_2, \frac{2\widetilde{M}}{\delta^{(1+\alpha)\beta}}\}$
we get
\begin{equation} \label{8}
|F_n(\xi)-F_n(\eta)| \leq  M|\xi-\eta|^{(1+\alpha)\beta},
\end{equation}
for all $\xi \in S$.  \\
Now, from Theorem $\ref{calpha}$, we have
\[ |\nabla f_n (r\eta)|\leq C(1-r)^{(1+\alpha)\beta-1}, \hspace{0.3cm} \text{$\forall r\in [0,1).$}\].
\\
As $f$ is quasiconformal mapping then
\[ \frac{\max\limits_{|h_1|=1} |f'(x)h_1|}{\min\limits_{|h_2|=1} |f'(x)h_2|}\leq K < \infty, \hspace{0.3cm} \text{$\forall x \in B$}.\]
Taking, $h_1=e_j$ and $h_2=e_n$, for $x=r\eta$ we have
\[|\nabla f_j(r\eta) | \leq K |\nabla f_n(r\eta) |\leq K\cdot C (1-r)^{(1+\alpha)\beta-1},\]
for all $j\in \{1,\ldots, n-1\}$.
\\
This implies \begin{equation} \label{012} |\nabla f_j(r\eta)|\leq C(1-r)^{(1+\alpha)\beta-1},\end{equation} where $C$ is a new global constant for all $j \in\{1,\ldots,n\},$ and all $r\in[0,1)$.
\\ We want to prove $(\ref{012})$ in ${B}$. Let $\eta_1\neq \eta$ be an arbitrary point on $S$ and $f(\eta_1)=q_1$. Let $L_{q_1}$ be the isometry that sends $q_1$ to $0$, with $x_n=0$ being the tangent plane of $L_{q_1}(\partial \Omega)$ at $L_{q_1}(q_1)=0$.

Let $L_{q_1}\circ f=\tilde{f}=(\tilde{f}_1, \ldots, \tilde{f}_n)$. Then $\tilde{f}$ has all the properties of the function $f$ with $\eta_1$ in place of $\eta$: at $\tilde{f}(\eta_1)=0$ the tangent plane of the surface $L_{q_1}(\partial \Omega)$ is  $x_n=0$ and $\tilde{f}(\eta_1)$ has a neighbourhood in $L_{q_1}(\partial \Omega)$ which can be expressed as a part of a graph of the form $(\ref{5})$. Using the same procedure, we conclude that
\[|\nabla \tilde{f}_j(r\eta_1)|\leq C(1-r)^{(1+\alpha)\beta-1},\] for all $j \in\{1,\ldots,n\},$ and all $r\in[0,1)$.  Constant $C$ is universal and it does not depend on $\eta_1$, because $\delta$ and $M$ are independent of the choice of $\eta \in S$.   \\
Since $f=L_{q_1}^{-1} \tilde{f}$, ($L_{q_1}^{-1}$ is also an isometry) we get
\[ f_j(\xi)=b_j+\sum_{k=1}^{n}a_{j,k}\tilde{f}_k(\xi), \]
$j\in \{1,\ldots,n\}$, so
\begin{equation} \label{013} \nabla f_j(\xi)=\sum_{k=1}^{n}a_{j,k}\nabla \tilde{f}_k(\xi), \end{equation}
where $\{a_{j,k}\}_{1\leq j,k\leq n}$ is an orthogonal matrix.
From $(\ref{013})$ we have:
\begin{equation} \label{009}
\begin{split}
 |\nabla f_j(\xi)| &\leq \sum_{k=1}^{n}|a_{j,k}||\nabla \tilde{f}_k(\xi)| \\
& \leq \left(\sum_{k=1}^{n} |\nabla \tilde{f}_k(\xi)|^2 \right)^{\frac{1}{2}}.
\end{split}
\end{equation}
In the last inequality it is used the Cauchy-Schwarz inequality and the orthogonality of matrix $\{a_{j,k}\}_{j,k=1}^{n}$. Taking $\xi=r\eta_1$ we get
\[  |\nabla f_j(r\eta_1)| \leq \sqrt{n}C (1-r)^{(1+\alpha)\beta-1}.   \]
As the point $\eta_1$ was arbitrary we conclude
\[ |\nabla f_j(x)|\leq C (1-r)^{(1+\alpha)\beta-1}, \hspace{0.3cm} \text{$r=|x|,$} \]
for all $x\in B,$ $j \in \{1, \ldots n\}.$ \\

From Lemma $\ref{015}$ it follows that $f_j \in C^{(1+\alpha)\beta}(\overline{B})$, for all $j\in \{1,\ldots,n\}$ and so $f \in C^{(1+\alpha)\beta}(\overline{B}).$

We could have chosen $\beta<\frac{1}{2}$ (by decreasing it, if necessary) so the numbers $(1+\alpha)^k\beta\neq 1$, for every $k$. As $1+\alpha>1$ there exists a unique integer $k_0$ such that $(1+\alpha)^{k_0} \beta<1$ and ($1+\alpha)^{k_0+1}\beta>1$.
Repeating the procedure, we get that $f \in C^{(1+\alpha)^2\beta}(\overline{B}), \ldots, C^{(1+\alpha)^{k_0}\beta}(\overline{B})$. Note that such procedure for two-dimensional setting and different purpose has been used in \cite{Nitsche} and in \cite{KalajLamel}.
Similar to $(\ref{6})$ it follows that $|F_n(\xi)-F_n(\eta)|\leq M |\xi-\eta|^{(1+\alpha)^{k_0+1}\beta}$, $\forall \xi \in S$. This time, using Theorem $\ref{calpha2}$ we obtain
\[ |\nabla f_n(r\eta)|\leq C, \hspace{0.5cm} \text{$\forall r\in [0,1)$}.\]
Using the same order of implications, first we get the same inequality for every $f_k$ on points $r\eta.$ Then, using the isometries, we get the inequality on every point of $B$ for a global constant $C$. This implies trivially, by mean value inequality, the Lipschitz continuity of function $f$ in $\overline{B}$.

\end{proof}

\section*{Acknowledgement} We are grateful to the anonymous referee for a number of corrections that have made this paper better.

\end{document}